\documentclass{amsart}
\usepackage{verbatim,amscd, mathabx}
\usepackage{amsmath,amssymb,amsfonts,amsthm}
\usepackage{amssymb}
\usepackage{pgf}
\usepackage{todonotes}
\usepackage{epsfig}
\usepackage{mathrsfs}
\usepackage{verbatim}
\newtheorem{theorem}{Theorem}[section]

\newtheorem{proposition}[theorem]{Proposition}

\newtheorem{corollary}[theorem]{Corollary}

\newcommand{\ode}{_{^\text{ODE}}}
\theoremstyle{definition}

\usepackage[utf8]{inputenc}

\theoremstyle{remark}

\numberwithin{equation}{section}

\newcommand{\iii}{{\vert\kern-0.25ex\vert\kern-0.25ex\vert }}

\newcommand{\R}{\mathbb{R}}
\newcommand{\N}{\mathbb{N}}

\newcommand{\essinf}{\mathop{\rm ess\ inf}}
\newcommand{\esssup}{\mathop{\rm ess\ sup}}

\begin{document}

\title[Decompositions of Nakano norms by ODE techniques]{Decompositions of Nakano norms\\ by ODE techniques}

\author{Jarno Talponen}
\address{University of Eastern Finland, Department of Physics and Mathematics, Box 111, FI-80101 Joensuu, Finland}
\email{talponen@iki.fi}

\date{\today}

\begin{abstract}
We study decompositions of Nakano type varying-exponent Lebesgue norms and spaces. These function spaces are 
represented here in a natural way as tractable varying-exponent $\ell^p$ sums of projection bands. The main results involve embedding the variable Lebesgue spaces to such sums, as well as the corresponding isomorphism constants. 

The main tool applied here is an equivalent variable Lebesgue norm which is defined by a suitable ordinary differential equation introduced recently by the author. We also analyze the effect of transformations changing the ordering of the unit interval on the values of the ODE-determined  norm.
\end{abstract}

\maketitle

\section{Introduction}
In this paper we analyze Nakano type varying-exponent $L^p$ norms and their decompositions. 
Alternatively, we study embedding results of the above Nakano spaces to some more tractable Banach spaces
which arise as varying $\ell^p$ summands of classical $L^p$ spaces. We approach this topic via an alternative equivalent norm on the Nakano space defined by the means of weak solutions to suitable ordinary differential equations (ODE). 
These can be easily analyzed by studying the properties of the corresponding ODEs and thus by virtue of equivalence of the norms we obtain strong estimates for Nakano type variable Lebesgue norms. 

Motivated by Nakano's work (\cite{Nakano}), we call the following special case of Luxemburg (\cite{lux}) norm on a Musielak-Orlicz space (\cite{mus}, cf. \cite{bo}) a Nakano norm:
\[\iii f\iii_{p(\cdot)} := \inf \left\{\lambda>0 \colon \rho(f/\lambda) \leq 1 \right\}\]
where 
\[\rho(g) :=\int \frac{1}{p(t)} |g(t)|^{p(t)}\ dt .\]
See \cite{Maligranda} for a historical account. 

The Luxemburg norm clearly arises as an application of the Minkowski functional. It is not that easy to analyze the 
Luxemburg norms, even the numerical value of the norm of a constant function cannot be calculated 
immediately based on the definition of the norm. Here we consider an alternative approach  to variable Lebesgue 
norms which is rather recursive than global.

The above weight $w(t)=\frac{1}{p(t)}$, appearing in Nakano's work, catches the eye
(see also discussion in \cite[Sect. 3.1]{book}). In a sense it involves the continuity properties of the mapping $\lambda \mapsto \rho(f/\lambda)$. In Section \ref{sect: CM} we provide further motivation and discussion. This weight also has a special role in our ODEs studied. These provide a novel approach to the varying-exponent Lebesgue spaces which do \emph{not} involve the Minkowski functional at all.

There is a vast literature involving inequalitites and embeddings on the variable Lebesgue spaces. For an important paper initiating the investigation of varying-exponent Sobolev spaces, see \cite{kovacik}. Different kinds of inequalities are in the core of the study of Lebesgue and Sobolev spaces. These are essential for instance in classical integral operators, harmonic analysis and spaces of analytic functions.
These spaces involve for instance the Poincar\'e-like inequalities carrying the names of Morrey, Nash and Gagliardo-Nirenberg-Sobolev. The norms of a Riesz operator and a potential are related by the Hardy-Littlewood-Sobolev lemma.
We refere to  \cite{UF,UFMP,CW,MSS} for samples of papers concentrating on these and other important inequalities. 
See also \cite{CFN,die0, DHH,Ho,RS} for works which rely on techniques with central inequalities.
All the mentioned inequalities deal with $L^p$ structures which is a very particular case in the realm of Banach spaces. The exponent $p$ may additionally vary locally but nevertheless these structures are detectable. Apart from the typical applications in mathematical analysis, the variable Lebesgue spaces have been recently an object of study in mathematical logic, see e.g. \cite{bej}. 

The main purpose of this paper is to provide natural tools for treating the $L^p$ structure of 
variable Lebesgue spaces and to demonstrate the usefulness of ODE-techniques in this connection.

Because of the way Nakano norms, and, more generally, Luxemburg norms are defined, it is not so easy to identify a priori the 'contribution' of different parts of the function to the norm; that is, how the restriction of the support of the function affects the norm. Namely, if $f \in L^{p(\cdot)} [0,1]$ (with any weight) and $\Delta \subset [0,1]$ is measurable, possibly disjoint from the support of $f$, it is not easy in general to evaluate efficiently the relationship between 
$\iii f \iii_{p(\cdot)}$ and $\iii f + 1_\Delta \iii_{p(\cdot)}$ quantitatively.

It seems reasonable to ask how the Nakano norm can be decomposed in a natural way. For instance, if $p(\cdot)$ has constant values, say, $p_1 , \ldots , p_n$, on some sets of positive measure, one might ask what is the quantitative relationship between the following Nakano norms:
\[\iii  f \iii_{p(\cdot)}\quad \text{and}\quad \iii  1_{p(\cdot) \neq p_1 , \ldots , p_n} f\iii  _{p(\cdot)}\ , \ 
\iii  1_{p(\cdot) = p_1} f\iii  _{p(\cdot)}\ ,\ \ldots ,\ \iii  1_{p(\cdot) = p_n} f\iii  _{p(\cdot)}?\]
If the correspondence is natural enough, it should not depend on $n$, since the measure space is atomless.
The relationship should be somehow formalized mathematically and the simplest way is stating
\[\iii  f\iii  _{p(\cdot)} \sim \iii  1_{p(\cdot) \neq p_1 , \ldots , p_n} f\iii  _{p(\cdot)} + \iii  1_{p(\cdot) = p_1} f\iii  _{p(\cdot)}+ \ldots + \iii  1_{p(\cdot) = p_n} f\iii  _{p(\cdot)},\]
which, unfortunately, turns out to be the wrong approach after a short reflection, unless $p(\cdot)=1$. A reasonable approach is replacing above the 
addition operations by suitable operations corresponding to $\ell^p$-summands applying the respective values of $p_i$. It may be instructive to observe that 
\[\|f\|_{p} = \left(\|1_\Delta f\|_{p}^p + \|1_{[0,1]\setminus \Delta} f\|_{p}^p \right)^{\frac{1}{p}}\]
for any $f\in L^p [0,1]$, $1\leq p<\infty$, and a measurable subset $\Delta \subset [0,1]$. Thus the correct operation in this simple case appears to be
\[a \boxplus_p b := (a^p + b^p )^{\frac{1}{p}},\quad a,b\geq 0,\] 
instead of the $+$ operation. This leads to the relationship
\begin{multline*}
\iii  f\iii  _{p(\cdot)} \sim
 (\ldots ((\iii  1_{p(\cdot) \neq p_1 , \ldots , p_n} f\iii  _{p(\cdot)} \boxplus_{p_1} 
\iii  1_{p(\cdot) = p_1} f\iii  _{p(\cdot)}) \boxplus_{p_2} \ldots\\
 \ldots \boxplus_{p_{n-1}} \iii  1_{p(\cdot) = p_{n-1}} f\iii  _{p(\cdot)} ) \boxplus_{p_n} \iii  1_{p(\cdot) = p_n} f\iii  _{p(\cdot)}.
\end{multline*}
We will prove in this paper several inequalities which substantiate the above equivalence as a simple case of a more 
general natural principle (see Theorem \ref{thm: main_decomp}). The obtained isomorphism constants do not depend on $n$. We also analyze the effect of changing the order in taking the operations above. 

The main tool is an equivalent ODE-determined Lebesgue norm introduced by the author in \cite{Tal}, extending 
the class of classical $L^p$ norms. This construction can be seen as a continuous version of a recursively constructed varying-exponent $\ell^p$ norm. It appeared in \cite{Sobczyk} as a remark and was later studied in \cite{talponen}, 
cf. \cite{kalton1, kalton2}. 
Unlike the Luxemburg style variable Lebesgue norm, the ODE-determined norm satisfies properly the H\"{o}lder inequality and the dualities are neat. The norm determining ODE is rather simple:
\[\varphi_f (0 )=0 ,\ \varphi_{f}'(t)=\frac{|f(t)|^{p(t)}}{p(t)}\varphi_{f} (t)^{1-p(t)}\quad \mathrm{for\ almost\ every}\ t\in [0,1].\]
It turns out that in some cases arguing by means of the above differential equation makes the analysis of the norms tractable.

Representing the variable Lebesgue norms by using projection bands is useful in the analysis of operators
on these spaces. Recall that a projection band is a subspace of functions supported on a given measurable subset, cf. \cite{LTII}.
For instance, the $p$-convexifications or $\ell^p$-type sums of functions are instrumental in 
\cite{DHH} and \cite{MSS}. We obtain natural decompositions of Lebesgue spaces which yield upper and lower norm
estimates for variable Lebesgue norms. This likely simplifies the analysis of the various operators acting on these spaces
and otherwise reduces norm estimates to a combination of well-understood `blocks' with classical $L^p$ norms.

\section{On variable Lebesgue spaces}
We rely heavily on the properties and the theoretical framework of ODE-determined norms appearing in \cite{Tal}.
See \cite{DE_book,book, FA_book, Maligranda, mus, rr} for other suitable background information. 

In what follows $f, p \in L^0 [0,1]$, i.e. $f, p\colon [0,1]\to \R$ are measurable functions, and $p(\cdot)\geq 1$.
We denote by $\overline{p}=\esssup p(\cdot)$ if $p(\cdot)$ is essentially bounded. Denote $a \boxplus_\infty b = \max(a,b)$ for $a,b\geq 0$.

We will study a special type of variable Lesbesgue norm, namely the Nakano norm:
\[\iii f\iii_{p(\cdot)} := \inf \left\{\lambda>0 \colon \int_{0}^1 \frac{1}{p(t)} \frac{|g(t)|^{p(t)}}{\lambda^{p(t)}}\ dt \leq 1 \right\}.\]

In the literature the variable Lebesgue spaces are considered over $n$-dimensional space, say, 
expressed in a rather high degree of generality, spaces of the type $L^{p(\cdot)} (\R^n, \mathrm{Bor}(\R^n), m_n , \omega)$ where $(\R^n, \mathrm{Bor}(\R^n), m_n )$ is a standard measure space involving the Lebesgue measure 
and the Borel $\sigma$-algebra on $\R^n$, 
$\omega \colon \R^n \to [0,\infty)$ is a measurable 'weight' function (possibly restricting the essential domain) and the modular is given by    
\[\rho(g) :=\int_{\R_n} \omega (t) |g(t)|^{p(t)}\ dm_n (t) .\]
Alternatively, one could apply in the integral a measure $\mu \colon \mathrm{Bor}(\R^n) \to \R$ defined by
\[\mu(A)=\int_{A} \omega (t) \ dm_n (t) ,\quad A \in \mathrm{Bor}(\R^n).\]
Note that this defines a $\sigma$-finite measure. We refer to \cite{Fremlin} for advanced measure theory.
Our methodology in this paper, being an application of ODEs, requires totally ordered measure spaces, essentially the case $n=1$. 
This feature appears to narrow down the applicability of these ODE-defined spaces. 

Unexpectedly, it turns out that this is not the case, at least in what comes to the order-isometric structure of these spaces, which is the topic here.

\begin{corollary}\label{cor: isom}
The dimension $n$ in $L^{p(\cdot)} (\R^n, \mathrm{Bor}(\R^n), m_n , \omega)$ is isometrically order-isomorphically redundant
due to suitable rearrangements of the underlying measure spaces. More precisely, there exists an isometric isomorphism
\[T\colon L^{p(\cdot)} (\R^n, \mathrm{Bor}(\R^n), m_n , \omega) \to L^{r(\cdot)} ([0,1], \mathrm{Bor}([0,1]), m_1 , \omega') \] 
having the form
\[(T f)[h(x)] = a(x) f[x] \quad \text{for}\ \mu\text{-a.e.}\ x\in \R^n \] 
where
\[ r(h(x)) =p(x)\quad \text{and} \quad    \omega' (h(x)) = \omega(x) \quad \text{for}\ \mu\text{-a.e.}\ x\in \R^n ,\]
$a\colon \mathrm{supp}(\mu)  \to (0,\infty)$ is a measurable function, 
and $h \colon \mathrm{supp}(\mu) \to [0,1]$ a bijection.
\end{corollary}
This follows form the results given in the Final Remarks (section \ref{sect: red}) where we also provide a comprehensive
description of the functions above.

\subsection{ODE-determined norms}

The main tool in analyzing the Nakano norms here is passing to a tractable norm, defined by means of an ODE, such that the new norm $\|\cdot\|_{L_{\mathrm{ODE}}^{p(\cdot)}}$ 
is equivalent to the Nakano norm:
\[\iii f\iii_{p(\cdot)}  \leq \|\cdot\|_{L_{\mathrm{ODE}}^{p(\cdot)}} \leq 2 \iii f\iii_{p(\cdot)}, \]
see \cite[Prop. 3.3]{Tal}. The ODE-determined varying-exponent Lebesgue class $L_{\mathrm{ODE}}^{p(\cdot)} [0,1]$ can in principle be defined for any measurable $p\colon [0,1]\to [1,\infty)$. However, if $p(\cdot)$ is essentially unbounded it may happen that the class fails to be a linear space. Therefore, for the purposes in this paper, we may restrict to the case where $p(\cdot)$ is essentially bounded, since the corresponding Nakano norms can in any case be approximated pointwise by suitably truncating supports. 
 
The strategy in \cite{Tal} is to design the ODE in such a way that its solution $\varphi_{f,p(\cdot)}$, corresponding to the function and the exponent, models the accumulation of the norm as follows:
\[\varphi_{f,p(\cdot)} (t) = \|1_{[0,t]} f\|_{p(\cdot)}, \]
so that in particular $\varphi_{f,p(\cdot)}(0)=0$ and $\varphi_{f,p(\cdot)}(1)$ becomes the definition for the new norm
$=\| f \|_{p(\cdot)}$. The ODE which defines the non-decreasing absolutely continous solution is
\begin{equation}\label{eq: def}
\varphi_f (0 )=0^+ ,\ \varphi_{f}'(t)=\frac{|f(t)|^{p(t)}}{p(t)}\varphi_{f} (t)^{1-p(t)}\quad \mathrm{for\ almost\ every}\ t\in [0,1].
\end{equation}
These are weak solutions in the sense of Carath\'eodory with a minor modification;
the asymptotic initial condition is to provide the uniqueness of the solution and it is useful in other ways as well.
These solutions are further discussed in \cite{Tal}. 

The ODE-determined variable Lebesgue class is defined as
\[ L_{\mathrm{ODE}}^{p(\cdot)} [0,1] = \left\{f \in L^0 [0,1] \colon \varphi_f \ \text{exists}\right\}\] 
where the solution to \eqref{eq: def} exists in the required sense. 

Recall from \cite{Tal} that the $\varphi_{f} (t)^{1-p(t)}$ part in the ODE makes it very stable, so that for positive initial values $x_0 >0$
the solutions are unique. The $0^+$ initial value solution in \eqref{eq: def} means that we take first all positive-inital-value-solutions, provided that they exist, and then take 
their limit pointwise as $x_0 \searrow 0$. This leads to a rather natural unique solution which solves \eqref{eq: def} for $0$ initial value. For a positive inital value $x_0>0$ 
and an essentially bounded exponent $p(\cdot)$ we have 
\[\varphi_{x_0 , f}'(t) \leq \frac{|f(t)|^{p(t)}}{p(t)}(x_0 )^{1-\overline{p}}\]
from which we can deduce the existence of the solution if 
\[\int_{0}^1  \frac{|f(t)|^{p(t)}}{p(t)}\ dm(t)<\infty.\]
The solutions can be constructed and approximated by using simple semi-norms (\cite{Tal}) which are motivated by variable $\ell^p$ spaces (see \cite{talponen}, cf. \cite{tal15}). 
The simple semi-norms here have a role analogous to the simple functions in the construction of the Lebesgue integral.

The absolute continuity of the solutions implies that $\varphi_f (1)<\infty$.
In some cases, for instance if $p(\cdot)$ is essentially bounded, this class becomes automatically a linear space. 
Then the \emph{solutions} define a complete norm (see \cite{Tal}):
\[\|f\|_{p(\cdot)} := \|f\|_{L_{\mathrm{ODE}}^{p(\cdot)}} := \varphi_f (1),\quad f\in L_{\mathrm{ODE}}^{p(\cdot)} [0,1] . \]

For a constant exponent case, $p(\cdot)=p \in [1,\infty)$, the above ODE \eqref{eq: def} becomes a separable one and solving it yields the classical $L^p [0,1]$ norm: $(\varphi_f (1))^p = \int_{0}^1 |f(t)|^p \ dt$. It is worth noting that unlike the usual Luxemburg type variable Lebesgue norms, the ODE-determined norms satisfy properly the H\"{o}lder inequality. 
The dualities work nicely as well, see \cite{Tal2}. 

\subsection{Some useful estimates}

Let $a\approx 1.76$ be the solution to $a^a =e$. This number satisfies that $\frac{b^x}{x}$ is increasing on $x\geq 1$ for all $b\geq a$. Namely, the constant $a$ satisfies that $\frac{d}{da} a^x |_{x=1} =1$. 
Let us recall the following useful fact from \cite{Tal}:
\begin{proposition}\label{prop: ineq}
Let $f,p\in L^0$ where $1\leq p(\cdot)$ and $r \in (1,\infty)$. The following inequalities hold whenever defined:
\begin{enumerate}
\item{$\frac{1}{1+a} \|1_{p(\cdot)\geq r}f\|_r  \leq \|1_{p(\cdot)\geq r} f\|_{p(\cdot)}$,}
\item{$\frac{1}{1+ae} \|1_{p_{1}(\cdot)\leq p_{2}(\cdot)}f\|_{p_1 (\cdot)}  \leq \|1_{p_1 (\cdot)\leq p_2 (\cdot)} f\|_{p_2 (\cdot)}$,}
\item{$\|f\|_{p(\cdot)} < e\|f\|_\infty $.}
\end{enumerate}
\end{proposition}

Although we are here mainly insterested in the relationship between Nakano and ODE-determined type variable Lebesgue 
norms, the results have also some bearing on the most typical Luxemburg type variable Lebesgue norms given by
\[\iii f \iii_ {\text{MO}, p(\cdot)} := \inf \left\{\lambda>0 \colon \int_{0}^1 \frac{|g(t)|^{p(t)}}{\lambda^{p(t)}}\ dt \leq 1 \right\},\]
denoted here after Musielak and Orlicz. Indeed, despite the weight in the Nakano norm, it is equivalent to the MO norm.
This is known (see e.g. \cite[(3.2.2)]{book}) and below we provide a  better isomorphism constant.
\begin{proposition}\label{prop: NMO}
Given a measurable $p\colon [0,1] \to [1,\infty)$ then $f \in L^{p(\cdot)}$ in the sense of Nakano norm if and only if 
the same holds in the sense of the above MO norm. Moreover, in these equivalent cases with $f \in L^{p(\cdot)}$ we have
\[  \frac{1}{a} \iii f \iii_ {\text{MO}, p(\cdot)}  \leq  \iii f \iii_ {p(\cdot)}  \leq   \iii f \iii_ {\text{MO}, p(\cdot)} .\]
\end{proposition}
\begin{proof}
We apply the following inequalities
\[\int_{0}^1 \frac{|f(t)|^{p(t)}}{(a \lambda)^{p(t)}}\ dt \leq
\int_{0}^1 \frac{1}{p(t)}\frac{|f(t)|^{p(t)}}{\lambda^{p(t)}}\ dt \leq \int_{0}^1 \frac{|f(t)|^{p(t)}}{\lambda^{p(t)}}\ dt \]
where the left inequality follows from the above property of the constant $a$.
\end{proof}

\newcommand{\ODE}{\mathrm{ODE}}

Below we consider a constant $b_{p}$ depending on 
$p$ defined as follows: $b_{1}=1$, $b_{\infty} = 2$ and $b_{p}$ for $1<p<\infty$ is the unique solution to 
\[b + b^{-p} =2,\quad  1<b < 2.\]  
Below we will apply the norm notation without distinguishing whether the $L_{\ODE}^{p(\cdot)}$ class is a linear space or not. The following result improves some estimates in \cite{Tal}.
\begin{proposition}\label{prop: Nak}
Let $f \in L_{\ODE}^{p(\cdot)}$. Then 
\[\iii f\iii_{p(\cdot)} \leq \|f\|_{p(\cdot)} \leq b_{\overline{p}} \iii f\iii_{p(\cdot)}.\]
\end{proposition}
\begin{proof}
Let $f \in L_{\ODE}^{p(\cdot)}$. These inequalities were proved in \cite{Tal}, except that a weaker version
of the right-hand side was shown with $2$ in place of $b_{p}$.

To check the latter inequality, we may assume without loss of generality that $\iii f\iii_{p(\cdot)} =1$. 
By using the Monotone Convergence Theorem we see that 
\[\int_{0}^1 \frac{1}{p(t)}|f(t)|^{p(t)}\ dt =1 .\]
Let $b\leq 2$ be the best constant for the latter inequality. 
Suppose that $\|f\|_{p(\cdot)} := \varphi_f (1) >1$ and that $t_0 \in (0,1)$ is such that 
$\varphi_f (t_0 )=1$. We will apply the well-known fact (see \cite{book}) that 
\[\iii f\iii_{p(\cdot)} \leq \left(\int_{0}^1 \frac{1}{p(t)}|f(t)|^{p(t)}\ dt \right)^{\frac{1}{\overline{p}}} \]
for every $f \in L^{p(\cdot)}$ (Nakano space), so that 
\[\iii 1_{[0,t_0 ]} f \iii_{p(\cdot)} \leq \left(\int_{0}^{t_0} \frac{1}{p(t)}|f(t)|^{p(t)}\ dt \right)^{\frac{1}{\overline{p}}}. \]
Since $\| 1_{[0,t_0 ]} f \|_{p(\cdot)} = 1$ we obtain by the choice of $b$ that 
\[b^{-1} \leq   \iii 1_{[0,t_0 ]} f \iii_{p(\cdot)}\]
thus 
\[b^{-\overline{p}} \leq  \int_{0}^{t_0} \frac{1}{p(t)}|f(t)|^{p(t)}\ dt .\]
Hence 
\[\int_{t_0}^1  \frac{1}{p(t)}|f(t)|^{p(t)}\ dt \leq 1- b^{-\overline{p}} .\] 
Note that
$\varphi_{f}^{1-p(t)} (t) \leq 1$ for $t \in [t_0 , 1]$. Thus
\[\varphi_{f}(1) = \int_{0}^1 \varphi_{f}' \ dt \leq 1 + \int_{t_0}^1  \frac{1}{p(t)}|f(t)|^{p(t)}\ dt  \leq 
2- b^{-\overline{p}} .\]
Since $\iii f\iii_{p(\cdot)}=1$ and $b$ is the best constant, this implies 
$b \leq 2- b^{-\overline{p}}$. Thus $b$ is dominated by the solution to 
\[b_{\overline{p}} + b_{\overline{p}}^{-\overline{p}} =2,\quad  1< b_{\overline{p}} < 2.\]  
\end{proof}

\begin{proposition}
Suppose that $p(\cdot)$ is essentially bounded or non-decreasing. Then the corresponding $L^{p(\cdot )}$ classes, in the ODE, Nakano and MO sense
coincide and the norms are mutually equivalent.
\end{proposition}
\begin{proof}
In the case with essentially bounded exponent $f \in L_{\mathrm{ODE}}^{p(\cdot)} [0,1]$ if and only if 
\[\int_{0}^1 \frac{1}{p(t)}|f(t)|^{p(t)}\ dt < \infty \]
see \cite[Prop. 3.8.]{Tal}. Combining this fact with Proposition \ref{prop: NMO} yields that the $L^{p(\cdot )}$ classes coincide.
The equivalence of the norms then follows from Proposition \ref{prop: NMO} and  Proposition \ref{prop: Nak}.

The case with non-decreasing exponent becomes an easy adaptation of the above argument, since in each proper initial segment of the unit interval the exponent is bounded.
\end{proof}

For $(a_i) \in \ell^p (\N)$, $1\leq p <\infty$, with $a_i \geq 0$ for all $i$, we define 
\[\bigboxplus^{p}_{i\in\N} a_i = \|(a_i)\|_{\ell^p}. \]
By inductively applying Proposition 2.1 in \cite{Tal2} we see the following fact.

\begin{proposition}\label{prop: multiboxplus}
If $1\leq r\leq s <\infty$ and $\{x_{N}\}_{N\in \N^{n+2}}$ is a family of non-negative numbers, then 
\begin{multline*}
\bigboxplus^{p_{n}}_{i_{n+2}\in\N} \ldots \bigboxplus^{p_{k+1}} _{i_{k+3}\in\N}
\bigboxplus^s _{i_{k+2}\in\N} \bigboxplus^r _{i_{k+1}\in\N}  \bigboxplus^{p_k} _{i_k \in\N} \ldots  \bigboxplus^{p_1} _{i_1 \in\N}
x_{N}\\ 
\leq \bigboxplus^{p_{n}}_{i_{n+2}\in\N} \ldots \bigboxplus^{p_{k+1}} _{i_{k+3}\in\N}
 \bigboxplus^r _{i_{k+1}\in\N} \bigboxplus^s _{i_{k+2}\in\N} \bigboxplus^{p_k} _{i_k \in\N} \ldots  \bigboxplus^{p_1} _{i_1 \in\N} x_{N}
\end{multline*}
where we consider $N=(i_1 , i_2 , \ldots , i_{n+2})$.

\end{proposition}
\qed

\section{Rearrangements}
In the Nakano space a simultaneous measure-preserving permutation of the exponent $p(\cdot)$ and the function $f$ does
not affect the value of the norm. This is not the case in an ODE-determined variable $L^p$ space, since the order 
of the arrangement affects the accumulation of the solutions $\varphi_f$. 

It is easy to see that 
\[a \boxplus_\infty (b \boxplus_1 c ) \leq  (a \boxplus_\infty b) \boxplus_1 c  \]
for all $a,b,c \geq 0 $, and, more generally, 
\[a \boxplus_r (b \boxplus_p c ) \leq  (a \boxplus_r b) \boxplus_p c \] 
holds for all $1\leq p \leq r < \infty$ and $a,b,c \geq 0$.
This inequality generalizes considerably, as we saw in Proposition \ref{prop: multiboxplus}. 

It appears natural to ask whether a similar conclusion holds for ODE-determined variable Lebesgue spaces.
Namely, if we have a simultaneous rearrangement of the exponent and the function, does the increasing (resp. decreasing) arrangement yield the minimal (resp. the maximal) value of the norm? If so, is the ratio of the maximal value and the minimal value bounded and by what constant?

\subsection{Heuristic motivation via an auxiliary transformation}
We will begin with a useful transformation. Suppose that $p\colon [0,1] \to [1,\infty)$ is a measurable function and 
$f \in L_{\ode}^{p(\cdot)}$. It follows from the equivalence of the ODE-determined norm $\|\cdot\|_{p(\cdot)}$ and the 
corresponding Nakano norm $\iii  \cdot\iii  _{p(\cdot)}$ that
\[\int_{0}^1 \frac{|f(t)|^{p(t)}}{p(t)}\ dt < \infty.\]

Let us assume that $f\neq 0$ a.e. We denote by $\varphi_f$ the norm-determining weak solution with initial value $0^+$.

Let us define an absolutely continuous increasing bijective transform $T \colon [0,1] \to [0, \alpha ]$ as follows
\[T(t) = \int_{0}^t \frac{|f(s)|^{p(s)}}{p(s)}\ ds ,\quad 0\leq t\leq 1\]
where $\alpha=T(1)$.

Define $\hat{\varphi} \colon [0, \alpha ] \to [0,\infty)$ and $\hat{p} \colon [0, \alpha] \to [1,\infty)$ by $\hat{\varphi}(t) := \varphi_f (T^{-1} (t))$ and $\hat{p} (t) := p(T^{-1} (t))$.
Observe that
\begin{multline*}
\hat{\varphi}'(t) = \frac{d}{dt} \varphi_f (T^{-1} (t)) = \left(\frac{d}{dt} T^{-1} (t)\right) \varphi_{f}' (T^{-1} (t))\\
= \left(\frac{|f(T^{-1} (t))|^{p(T^{-1} (t))}}{p(T^{-1} (t))} \right)^{-1}
\frac{|f(T^{-1} (t))|^{p(T^{-1} (t))}}{p(T^{-1} (t))} \varphi_{f}^{1-p(T^{-1} (t) )}(T^{-1} (t))
=\hat{\varphi}^{1-\hat{p} (t)} (t).
\end{multline*}

This explains heuristically why decreasing (resp. increasing) arrangement of the exponent yields the greatest (resp. the least) norm. Namely, due to the exponent $1-\hat{p} (t)\leq 0$ on the right-hand side, we obtain large values for 
$\hat{\varphi}'$ when: 
\begin{itemize}

\item $\hat{p}$ is large for small values of $\hat{\varphi}$, 

\item $\hat{p}$ is small for large values of $\hat{\varphi}$.

\end{itemize}

\subsection{Quantitative estimates for rearranged norms}

By a measure-preserving transform on the unit interval we mean a bijection $\pi \colon [0,1] \to [0,1]$ 
such that for each Lebesgue measurable subset $E \subset [0,1]$ the preimage $\pi^{-1} (E)$ is Lebesgue measurable 
and $m(\pi^{-1} (E))=m(E)$. See \cite[Vol III]{Fremlin} for discussion on representations of 
measure algebra isomorphims.

\begin{theorem}\label{thm: rearrangement}
Let $p\colon [0,1]\to [1,\infty)$ be a measurable function. Let $\pi \colon [0,1] \to [0,1]$ be a measure-preserving transform. Suppose that $f \in L_{\ODE}^{p(\cdot)}$ and 
$ f\circ \pi^{-1} \in L_{\ODE}^{p\circ \pi^{-1}}$. 
Then 
\[\frac{1}{b_{\overline{p}}}\|f \|_{p(\cdot)} \leq \|f\circ \pi^{-1} \|_{p\circ \pi^{-1}} \leq b_{\overline{p}} \|f \|_{p(\cdot)}.\]
Moreover, if $\pi_1$ and $\pi_2$, as above, are chosen such that $p\circ \pi^{-1}_1$ (resp. $p\circ \pi^{-1}_2$)
is an increasing (resp. decreasing) rearrangement of $p$, then 
\[ \|f\circ \pi_{1}^{-1} \|_{p\circ \pi_{1}^{-1}} \leq \|f \|_{p(\cdot)} \leq  \|f\circ \pi_{2}^{-1} \|_{p\circ \pi_{2}^{-1}}.\] 
\end{theorem}

We note that the constant $b_\infty =2$ above is the optimal one for the essentially unbounded $p(\cdot)$ case. This is seen by studying the norms of the constant function 
$\mathbf{1}$ considered in the natural way in $L^{p} (0,\frac{1}{\log p}) \oplus_{1} L^1 (\frac{1}{\log p},1)$, and, alternatively, in 
$L^1  (0,\frac{1}{\log p}) \oplus_{p} L^{p}(\frac{1}{\log p},1)$. Then 
\[\| \mathbf{1}\|_{L^{p} (0,\frac{1}{\log p}) \oplus_{1} L^1 (\frac{1}{\log p},1)}\to 2\quad \text{and}\quad \| \mathbf{1}\|_{L^1  (0,\frac{1}{\log p}) \oplus_{p} L^{p}(\frac{1}{\log p},1) }\to 1\quad \text{as}\ p\nearrow \infty.\]
The latter equation line in Theorem \ref{thm: rearrangement} tells us that the value of the norm of $\mathbf{1}$ is between $1$ and $2$ for any rearrangement of the above simple functions $p(\cdot )$.

\begin{proposition}
If $p(\cdot)$ is increasing then $\|\mathbf{1}\|_{p(\cdot)} \leq 1$ and if $p(\cdot)$ is decreasing then 
$\|\mathbf{1}\|_{p(\cdot)} \geq 1$. Moreover, if $p(\cdot )$ is additionally (essentially) a non-constant function then the inequalities hold strictly.
In particular, if $p(\cdot )$ is monotone and $\|\mathbf{1}\|_{p(\cdot)} =1$ then $p(\cdot )$ is a constant function.
\end{proposition}
\begin{proof}
Suppose that $1\leq p_1 \leq p_2 \leq \ldots \leq p_n$ and $0<a_1 , a_2 , \ldots , a_n  <1$ such that
$a_1 + \ldots + a_n \leq 1$.
Then
\[(a_{1}^{\frac{1}{p_1}} \boxplus_{p_2} a_{2}^\frac{1}{p_2} )\boxplus_{p_3} a_{3}^{\frac{1}{p_3}} = ((a_{1}^\frac{p_2}{p_1} + a_{2})^{\frac{p_3}{p_2}}+a_{3})^{\frac{1}{p_3}} \leq (a_1 + a_2 + a_3 )^{\frac{1}{p_3}} .\]
Similarly
\[(\ldots (a_{1}^{\frac{1}{p_1}} \boxplus_{p_2} a_{2}^{\frac{1}{p_1}} )\boxplus_{p_3} \ldots )\boxplus_{p_n} 
a_{n}^{\frac{1}{p_n}} \leq (a_1 + \ldots + a_n )^{\frac{1}{p_n}} \leq 1.\]
An approximation argument with semi-norms (recall in \cite{Tal}) then yields that if $p(\cdot)$ is increasing then $\|\mathbf{1}\|_{p(\cdot)} \leq 1$. 
Similarly we see the other direction.

To check the strict inequality, let $\varphi$ and $\psi$ be the solutions corresponding to decreasing and increasing rearrangements of $p$. Suppose that $\varphi (1)\leq \psi(1)$ and 
let $t_0 <1$ be the smallest number such that $\varphi (t) \leq \psi(t)$ for $t_0 \leq  t \leq 1$.
Then $\varphi' (t) \geq \psi' (t)$ for $t_0 \leq  t \leq 1$ and the inequality is strict in some segment
$[s,1]$ if $p$ is non-constant. 
\end{proof}

\begin{proof}[Sketch of proof of Theorem \ref{thm: rearrangement}]
To justify the first inequality line, it suffices to prove the right hand inequality, as the other estimate then follows.
We apply the equivalence of ODE-determined norms and Nakano norms together with the rearrangement invariance of the 
Nakano norms. That is, one may rearrange the function and the exponent in a similar fashion without affecting the norm. Thus we obtain 

\[\|f\circ \pi^{-1} \|_{p\circ \pi^{-1}} \leq b_{\overline{p}} \iii f\circ \pi^{-1} \iii_{p\circ \pi^{-1}}
=b_{\overline{p}} \iii f \iii_{p} \leq b_{\overline{p}} \|f \|_{p}.\]

Towards the second claim, it follows as a special case from Proposition \ref{prop: multiboxplus} that 
\[(a \boxplus_p b) \boxplus_r c \leq (a \boxplus_r c) \boxplus_p b,\quad \text{for\ all}\ a,b,c>0,\ 1\leq p \leq r .\]
We may apply the above fact recursively for simple semi-norms, changing the places of $2$ successive summands, $1$ pair at a time, in checking the following claim:
\begin{multline*}
|(0,f_{j_2} , f_{j_3},\ldots , f_{j_n})|_{(L^{p_{1}}(\mu_{1} ) \oplus_{p_{j_2}}  L^{p_{j_2}}(\mu_{j_2} ))\oplus_{p_{j_3} }\ldots \oplus_{p_{j_n} } L^{p_{j_n} }(\mu_{j_n} )}\\ 
\leq |(0 , f_2 , \ldots , f_n )|_{(L^{p_1}(\mu_1 ) \oplus_{p_2}  L^{p_2}(\mu_2 ))\oplus_{p_3 }\ldots \oplus_{p_n } L^{p_n }(\mu_n )} \\
\leq 
|(0,f_{k_2}, f_{k_3}, \ldots , f_{k_n})|_{(L^{p_{1}}(\mu_{1} ) \oplus_{p_{k_2}}  L^{p_{k_2}}(\mu_{k_2} ))\oplus_{p_{k_3} }\ldots \oplus_{p_{k_n} } L^{p_{k_n} }(\mu_{k_n} )}
\end{multline*}
where we assume $\mu_1 ([0,1])=0$ for technical reasons and $\{p_{j_i}\}_{i=2}^n$ 
(resp. $\{p_{k_i}\}_{i=2}^n$) is increasing (resp. decreasing) permutation of $\{p_i \}_{i=2}^n$.

To extend the above observation to the general setting, with decreasing exponent and maximal norm, we approximate $p(\cdot)$ by simple 
seminorms $N_n$ such that $\tilde{p}_{N_n} \nearrow p(\cdot)$ in measure. It is known that then also 
$N_n (f)\to \|f\|_{p(\cdot)}$ for any $f\in L^{p(\cdot)}$. It suffices to consider essentially bounded $p$ and $f$.

Let $\pi_n$ be simple measurable measure-preserving transformations such that $\tilde{p}_{N_n} \circ  \pi_{n}^{-1}$ become decreasing. Let $\psi_n$ be the corresponding solutions.
Note that by the above observations regarding the arrangements of the simple semi-norms, we have 
$N_n (f) \leq \psi_n (1)$. 

Suppose, as in the assumptions, that $\pi$ is a measurable measure-preserving transform such that $p\circ \pi^{-1}$ is decreasing. Let $\varphi_0 $ be the corresponding solution and recall that the solution $\varphi_0$ is (absolutely) continuous. 

Then for any $p_1 <p_2 $ and $\varepsilon>0$ we see by using the essential boundedness of $f$ and $p(\cdot)$ that for sufficiently large $n$ we have
\begin{equation}\label{eq: interv}
\int_{\{t:\ p_1 < p\circ  \pi_{n}^{-1}(t)< p_2\}} \psi'_n \ dt \leq  \int_{\{t:\ p_1 < p\circ \pi^{-1}(t)< p_2\}} \varphi'_0 \ dt + \varepsilon
\end{equation}
\begin{equation}\label{eq: if}
\text{if}\ \psi_n (t_0 ) \geq \varphi_0 (t_0 )\ \text{for}\ t_0 =\essinf \{t:\ p_1 < p\circ \pi^{-1}(t)< p_2\}.
\end{equation}
 
Indeed, by considering the distribution function $s\mapsto m(\{t\colon p(t)\geq s\})$ we obtain that  
\[\essinf \{t:\ p_1 < p\circ \pi_{n}^{-1}(t)< p_2\} \to \essinf \{t:\ p_1 < p\circ \pi^{-1}(t)< p_2\},\]
\[\esssup \{t:\ p_1 < p\circ \pi_{n}^{-1}(t)< p_2\} \to \esssup \{t:\ p_1 < p\circ \pi^{-1}(t)< p_2\}\]
as $n\to\infty$. 
It follows that 
\[\int_{\{t:\ p_1 < p\circ \pi^{-1} (t)< p_2\}} \left|\frac{|f\circ \pi_{n}^{-1}(t)|^{p\circ \pi_{n}^{-1}(t)}}{p\circ \pi_{n}^{-1}(t)} - \frac{|f\circ \pi^{-1}(t)|^{p\circ \pi^{-1}(t)}}{p\circ \pi^{-1}(t)} \right| dt \to 0 \]
as $n\to\infty$ by Lusin's theorem. 
Then \eqref{eq: interv} follows by including the respective terms $\psi_{n}^{1- p\circ \pi_{n}^{-1}}$ and 
$\varphi_{0}^{1- p\circ \pi^{-1}}$ according to \eqref{eq: if} (roughly $\psi_{n}^{1- p_1} \leq \varphi_{0}^{1- p_1}$ in the subset under consideration). Then
\begin{multline*}
\limsup_{n\to\infty} \psi_{n} (1) = \limsup_{n\to\infty} \int_{0}^1 \frac{|f\circ \pi_{n}^{-1}(t)|^{p\circ \pi_{n}^{-1}(t)}}{p\circ \pi_{n}^{-1}(t)} \psi_{n}^{1- p\circ \pi_{n}^{-1}}\ dt\\
\leq \int_{0}^1  \frac{|f\circ \pi^{-1}(t)|^{p\circ \pi^{-1}(t)}}{p\circ \pi^{-1}(t)} \varphi_{0}^{1- p\circ \pi^{-1}} \ dt = \varphi_{0} (1). 
\end{multline*}
We conclude that 
\[\|f\|_{p(\cdot)} = \lim_{n\to\infty} N_n (f) \leq \varphi_0 (1) .\]
\end{proof}

\section{Decompositions}

Next we arrive to the main topic of this paper which involves representing Nakano spaces in a natural tractable way
as summands of their projection bands. We will apply the observations from the previous section.

The decompositions of classical $L^p$ spaces provide some insight on what to expect, albeit, in some ways, too simplistic a view in variable Lebesgue space setting. 
For instance, in the classical setting we have  
\[L^p [0,1] = L^p [0, 1/2 ] \oplus_p L^p [1/2, 1]\]
and 
\[\|f\|_p \leq \|f\|_r ,\quad p\leq r,\]
therefore 
\[\|1_{[0,\frac{1}{2}]} f\|_p \boxplus_r \|1_{[\frac{1}{2},1]} f\|_r \leq \|f\|_r .\]
However, only a quasi-version of the above inequality holds in the variable Lebesgue space setting, and, in fact, the 
norm is not monotone with respect to the exponents $p(\cdot)$.

\begin{theorem}\label{thm: main_decomp}
Let $p(\cdot )$ be essentially bounded and $f \in L_{\ODE}^{p(\cdot)}$. Let $1= r_1 < r_2 < \ldots < r_n =\overline{p}$. Then 
\begin{multline*}
\frac{1}{ 2 (1+ae)} (\iii  1_{p^{-1}(r_{n-1} , r_n ]}f \iii  _{r_{n-1}} \boxplus_{r_{n-2}} \iii  1_{p^{-1}(r_{n-2} , r_{n-1}]}f \iii  _{r_{n-2}} ) \boxplus_{r_{n-3}} \ldots \boxplus_{r_1}  \iii  1_{p^{-1}[r_{1}, r_2 ]}f \iii  _{r_1}\\
\leq \iii  f\iii  _{p(\cdot)}
\leq   2 (1+ae) (\iii  1_{p^{-1}[r_1 , r_2 )}f \iii  _{r_2} \boxplus_{r_3} \iii  1_{p^{-1}[r_2 ,r_3 )}f \iii  _{r_3 })\boxplus_{r_4} \ldots 
\boxplus_{r_{n}}  \iii  1_{p^{-1}[r_{n-1},r_n ]}f \iii  _{r_{n} }.
\end{multline*}
More generally, if $\Delta_i$, $i=1,\ldots , n$, form a measurable decomposition of $[0,1]$, and 
$r_i = \essinf_{\Delta_i}p$, $s_i = \esssup_{\Delta_i}p$, then 
\begin{multline*}
\frac{1}{ 2 (1+ae)} (\iii  1_{\Delta_1}f \iii  _{r_1 } \boxplus_{r_2} \iii  1_{\Delta_2}f \iii  _{r_2 })\boxplus_{r_3} \ldots 
\boxplus_{r_{n}}  \iii  1_{\Delta_n }f \iii  _{r_{n} } 
\leq \iii  f\iii  _{p(\cdot) }\\ 
\leq  2 (1+ae) (\iii  1_{\Delta_1}f \iii  _{s_1 } \boxplus_{s_2} \iii  1_{\Delta_2}f \iii  _{s_2 })\boxplus_{s_3} \ldots 
\boxplus_{s_{n}}  \iii  1_{\Delta_n }f \iii  _{s_{n} } .
\end{multline*}
Moreover, 
\begin{multline*}
\frac{1}{b_{\overline{p}}}(\|1_{\Delta_1}f \| \boxplus_{s_2} \|1_{\Delta_2}f \|)\boxplus_{s_3} \ldots 
\boxplus_{s_{n}}  \|1_{\Delta_n }f \| \leq \|f\| _{p(\cdot)}\\
\leq  b_{\overline{p}}(\|1_{\Delta_1}f \| \boxplus_{r_2} \|1_{\Delta_2}f \|)\boxplus_{r_3} \ldots 
\boxplus_{r_{n}}  \|1_{\Delta_n }f \| 
\end{multline*}
and
\begin{multline*}
\frac{1}{b_{\overline{p}}} (\iii  1_{\Delta_1}f \iii   \boxplus_{s_2} \iii  1_{\Delta_2}f \iii  )\boxplus_{s_3} \ldots 
\boxplus_{s_{n}}  \iii  1_{\Delta_n }f \iii   \leq \iii  f\iii _{p(\cdot)}  \\ 
\leq  b_{\overline{p}} (\iii  1_{\Delta_1}f \iii   \boxplus_{r_2} \iii  1_{\Delta_2}f \iii  )\boxplus_{r_3} \ldots 
\boxplus_{r_{n}}  \iii  1_{\Delta_n }f \iii   .
\end{multline*}
\end{theorem}

Admittedly, the formulas appear complicated in the general case with $n$ fixed constant exponents. The point here is that the multiplicative constants above do not depend on $n$.
Note that in view of Theorem \ref{thm: rearrangement} the ordering of $r_i$ in the first part of the statement is chosen to be 'the worst possible', so that the given multiplicative 
constants apply to all possible orderings. The illustrate the first claim, we have
\[\frac{1}{12} \left(\iii  1_{p(\cdot) \geq r} f \iii_{r} + \iii  1_{p(\cdot) <r }f \iii_{1}\right)  
\leq  \iii  f \iii_{p(\cdot)} 
\leq 12 \left(\iii  1_{p(\cdot) <r }f \iii_{r}   \boxplus_{\overline{p}}  \iii  1_{p(\cdot) \geq r} f \iii_{\overline{p}} \right).\]

\begin{proof}[Proof of Theorem \ref{thm: main_decomp}]
To prove the first inequality, we may apply a measure-preserving rearrangement of the unit interval such that $p$ becomes decreasing under the new ordering. Indeed, by an approximation argument employing Lusin's theorem 
we may reduce to the case where $p(\cdot)$ is continuous and apply methods similar to the proof 
of Theorem \ref{thm: rearrangement}.  

The Nakano norm is invariant under simultaneous measure-preserving rearrangements of the function and the exponent, that is,
\[\iii f \circ h \iii_{p\circ h} = \iii f \iii_{p} \]
based on the change of variable 
\[\int_{0}^1 \frac{1}{p(t)} \frac{|f(t)|^{p(t)}}{\lambda^{p(t)}}\ dm(t) =  \int_{0}^1 \frac{1}{p(h(t))} \frac{|f(h(t))|^{p(h(t))}}{\lambda^{p(h(t))}}\ dm(t) .\]
Here $h\colon [0,1] \to [0,1]$ is a mapping such that $h(A)$ is measurable if and only if $A$ is measurable and 
$m(h(A))=m(A)$ for each measurable $A$. The Radon-Nikodym derivative satisfies $\frac{d(m \circ h)}{dm}\equiv 1$ due to the measure-preserving property and hence can be disregarded.

Therefore, we may, without loss of generality, assume that $p$ is decreasing.

Let us choose a varying-exponent $\tilde{p}$ corresponding to the space 
\[(L^{r_{n-1}} (p^{-1}(r_{n-1} , r_n ]) \oplus_{r_{n-2}} L^{r_{n-2}} (p^{-1}(r_{n-2} , r_{n-1} ])) \oplus_{r_{n-3}} \ldots \oplus_{r_{1}} L^{r_1 }  (p^{-1}[r_{1} , r_{2}]) .\] 
Recall Propositions \ref{prop: ineq} and \ref{prop: Nak}. Then 
\begin{multline*}
(\iii  1_{ p^{-1}(r_{n-1} , r_n ]}f \iii  _{r_{n-1}} \boxplus_{r_{n-2}} \iii  1_{p^{-1}(r_{n-2} , r_{n-1} ]}f \iii  _{r_{n-2}})\boxplus_{r_{n-3}} \ldots \boxplus_{r_{1}}  \iii  1_{p^{-1}[r_{1},r_2 ]}f \iii  _{r_1}\\ 
\leq  (\|1_{p^{-1}(r_{n-1} , r_n ]}f \|_{r_{n-1}} \boxplus_{r_{n-2}} \|1_{p^{-1}(r_{n-2} , r_{n-1} ]}f \|_{r_{n-2}})\boxplus_{r_{n-3}} \ldots \boxplus_{r_{1}}  \|1_{p^{-1}[r_{1},r_2 ]}f \|_{r_1}\\
=\|f\|_{\tilde{p} (\cdot)}\leq  (1+ae) \|f\|_{p (\cdot)}  \leq 2 (1+ae) \iii  f\iii  _{p(\cdot)}.
\end{multline*}
The right-hand inequality is seen in a similar manner.
\\

{\noindent \it The second part of the statement} is also seen in a similar way by first rearranging the unit interval such that the images of $\Delta_i$ become successive.
\\

{\noindent \it The third claim} follows inductively by considering first the case with $2$ summands only:
\[\|1_{\Delta_1}f \| \boxplus_{r} \|1_{\Delta_2}f \| \geq \|1_{\Delta_1 \cup \Delta_2} f \|\]
for the right-hand side.
The argument uses approximating simple semi-norms and the following observations, where $r\leq \min_k p_k$:
\begin{multline*}
( \ldots (a_1 \boxplus_{p_1} a_2) \boxplus_{p_2} \ldots \boxplus_{p_k} a_k )\boxplus_{p_{k+1}} a_{k+1} )\boxplus_{p_{k+2}} a_{k+2} \ldots \\
\leq ( \ldots (a_1 \boxplus_{p_1} a_2) \boxplus_{p_2} \ldots \boxplus_{p_k} a_k )\boxplus_{r} a_{k+1} )\boxplus_{p_{k+2}} a_{k+2} \ldots\\
\leq ( \ldots (a_1 \boxplus_{p_1} a_2) \boxplus_{p_2} \ldots \boxplus_{p_k} a_k ) \boxplus_{r} 
( a_{k+1}  \boxplus_{p_{k+2}} a_{k+2} \ldots )
\end{multline*}
Indeed, the last inequality follows from Proposition \ref{prop: multiboxplus} by putting 
\begin{multline*}
x_{1,1,1,1,1,\dots}= a_1,\  x_{2,1,1,1,1,\dots} = a_2,\  x_{1,2,1,1,1,\dots} = a_3,\  x_{1,1,2,1,1,\dots} = a_4,\\
x_{1,1,1,2,1,\dots} = a_5, \ldots , x_{1,1,1,\ldots , 1, 2,1,1,\ldots}=a_k,
\end{multline*}
\[x_{1,1,1,\ldots , 1, 2,2,1,1,\ldots}=a_{k+1},\ x_{1,1,1,\ldots ,1,  2,1,2,1,1,\ldots}=a_{k+2},\ x_{1,1,1,\ldots ,1,  2,1,1,2,1,\ldots}=a_{k+3},\ldots \]
and $0$ for other entries. The $\bigboxplus_r$ operation can be inductively moved to be the first operation on the left, thus producing the required inequality.

If the intervals are successive, then we are done, otherwise we apply Theorem 
\ref{thm: rearrangement} which involves the constant $b_{\overline{p}}$.
The left-hand inequality is seen similarly.
\\

{\noindent \it The last claim} is verified by using the previous claim with successive intervals and the invariance property of the 
Nakano norm together with the inequality
\[ \iii  f\iii  _{p(\cdot)} \leq \|f\|_{p(\cdot)}\leq  b_{\overline{p}}  \iii  f\iii  _{p(\cdot)} \]
from Proposition \ref{prop: Nak}.
\end{proof}


\section{Final Remarks}

\subsection{Reduction to the $n=1$ case}\label{sect: red}

Next we will give two results which yield Corollary \ref{cor: isom} in the beginning of the paper.
We use the same notations. 

\begin{theorem}\label{thm: identification}
Consider the setting of Corollary \ref{cor: isom} with $\mu(\R^n)>0$. Then there is a bijection 
$h\colon \mathrm{supp}(\mu) \to I$, where $I=\left[0,\mu(\R^n)\right)$, such that 
\[(T f)[h(x)] = f[x] \quad \text{for}\ \mu\text{-a.e.}\ x\in \R^n \] 
defines an isometric isomorphism
\[T\colon L^{p(\cdot)} (\R^n, \mathrm{Bor}(\R^n), m_n , \omega) \to L^{r(\cdot)} (I,\mathrm{Bor}\left(I\right) , m_1 , \omega') \] 
where 
\[r(h(x)) =p(x)\quad \text{and} \quad     \omega' (h(x)) = \omega(x) \quad \text{for}\ \mu\text{-a.e.}\ x\in \R^n .\]
\end{theorem}

\begin{proof}
Let us consider completed versions of both the measure spaces, $(\R^n , \Sigma , \mu)$ and 
$(I, \Sigma' , m)$.

The heart of the argument is Maharam's celebrated classification of measure algebras and their representation theory.
The topic of measure algebras is rather involved and it cannot be covered here in a self-contained fashion, thus 
we refer to \cite[Vol III]{Fremlin} for a detailed discussion. Measure algebras $(\mathfrak{A},\overline{\nu})$ 
can be realized (via Stone spaces) as quotients of suitable measure spaces $(\Omega,\mathcal{F},\nu)$,
\[\mathfrak{A} = \mathcal{F} / \mathrm{Null}(\nu),\]
where $\mathrm{Null}(\nu)$ is the sub-$\sigma$-ring of $\mathcal{F}$ consisting of sets $A$ such that $\nu(A)=0$.
Then $\mathfrak{A}$ becomes a $\sigma$-complete Boolean algebra in a natural way and a $\sigma$-additive mapping 
$\overline{\nu} \colon \mathfrak{A} \to [0,\infty]$ is canonically defined by $\overline{\nu} ([A])=\mu(A)$.
Note that $\Omega$ above has hardly any role in the measure algebras. However, when it does play a significant role,
this tends to be subtle.

Consider the probability spaces $([0,1]^n , \mathrm{Bor}([0,1]^n),m_n )$ and let $\mathfrak{A}$ and $\mathfrak{B}$ be the corresponding measure 
algebras with $n=3$ and $n=1$, respectively. These measure algebras are so-called Maharam homogenous and Maharam's
classification states that they are Boolean isomorphic to 
\[\{0,1\}^\omega \] 
normal form measure algebra where $\omega$ is the least infinite cardinal. This particular measure algebra is the classification of the above measure algebras and is generally known as 
\emph{the measure algebra}. More generally, the number $n \in \N$ does not affect the Maharam classification and this is crucial here.
The measure algebra corresponds to the standard probability space; the notation is instructive in this regard
in the sense that it can be modeled by an i.i.d. sequence of tosses of fair coins. Considering dydadic decompositions of, say, 
the unit square, and using it to model a sequence of independent coins gives some insight, why, modulo null sets, $n$ is irrelevant. 

Taking a quotient with respect to the null sets loses lots of information and this is the very reason why a general and elegent classification of measure algebras is possible, that is, with infinite cardinals $\kappa$ in place of $\omega$ in the homogenous case. Here we require a full isomorphism between measure spaces, rather than an induced Boolean isomorphism between the corresponding measure algebras. 

Measure spaces $(X, \mathcal{F}_1 ,\nu_1 )$ and $(Y, \mathcal{F}_2 , \nu_2 )$ are said to be isomorphic if there is a bijection $h\colon X \to Y$ such that 
$f(A) \in \mathcal{F}_2$ if and only if $A \in \mathcal{F}_1$ and then $\nu_2 (h(A))=\nu_1 (A)$ holds.
 
We will apply the following isomorphism result (see \cite[Vol. III, Thm. 344I]{Fremlin}):
Let $(X, \mathcal{F}_1 ,\nu_1 )$ and $(Y, \mathcal{F}_2 , \nu_2 )$ be atomless, perfect, complete, strictly localizable, countably separated measure spaces of the same non-zero magnitude. Then they are isomorphic. 

This of course implies that their measure algebras are isomorphic by a Boolean isomorphism induced by the measure space isomorphism. 
 Let us comment on these notions:
\begin{enumerate}
\item A measure space $(X, \mathcal{F} ,\nu)$ is atomless if for each $A \in   \mathcal{F}$, $\nu (A)>0$, there are disjoint $B,C \in  \mathcal{F}$ with $B\cup C = A$ and $\nu (B)>0$ and $\nu (C)>0$.
\item Given a topological space $X$, a Borel measure space $(X, \mathcal{F} ,\nu)$ is perfect if for each 
$A \in \mathcal{F}$ there are Borel sets $A_1 , A_2 \subset X$ such that $A_1 \subset A \subset A_2$ and 
$\mu(A_2 \setminus A_1 )=0$.  
\item A measure space is $(X, \mathcal{F} ,\nu)$ is $\sigma$-finite if there are $A_n \in \mathcal{F}$, $n\in\N$, 
with $\mu(A_n )<\infty$ and $X= \bigcup_n A_n$. This property implies strict localizability.
\item A measure space $(X, \mathcal{F} ,\nu )$ is complete if $A\subset B \in \mathcal{F}$ with $\nu (B)=0$ implies $A \in \mathcal{F}$.
\item A measure space $(X, \mathcal{F} ,\nu )$ is countably separated if there is a countable set 
$\mathcal{A} \subset  \mathcal{F}$ separating the points of $X$ in the sense that for any distinct 
$x, y \in X$ there is an $E \in \mathcal{A}$ containing one but not the other. 
(Of course this is a property of the structure $(X, \mathcal{F} )$  rather than of $(X, \mathcal{F} ,\nu )$.)
\item Two measure spaces have the same magnitude (e.g.) if they have the same total measure or if they both have infinite total measure 
and are $\sigma$-finite. 
\end{enumerate}
 
It is well-known that a completed Lebesgue measure space has these properties (regardless of the dimension $n$).
Then it is easy to see from the construction of the spaces $(\R^n , \Sigma , \mu)$ and $(I, \Sigma' , m)$ 
with $\mu(\R^n )=m(I)$ that they satisfy these conditions as well. Thus, according to the above measure space isomorphism result, 
\[\left(\R^n , \Sigma , \mu\right) \simeq \left(I, \Sigma' , m\right) \]
holds in the above sense and let $h \colon \R^n \to I$ be the bijection involved in the 
isomorphism. 
Now, if $p(\cdot)$, $\omega(\cdot)$ and $f\in L^{p(\cdot)} (\R^n, \mathrm{Bor}(\R^n), m_n , \omega) $ are given, 
then 
\[r(t) =p(h^{-1}(t)),\   \omega' (t) = \omega(h^{-1}(t)),\ (T f)[t]=f[h^{-1}(t)]\]
become measurable functions since $h^{-1}$ maps the sets in $\Sigma'$ to sets in $\Sigma$. 
By using the fact that $h$ is $\mu$-$m$-measure-preserving, we obtain by a change of variable for every 
$\lambda >0$  that 
\begin{multline*}
\int_{\R^n} \omega (x)\left| \frac{f(x)}{\lambda}\right|^{p(x)} \ d\mu(x)= 
\int_{ I } \omega (h^{-1}(t))\left| \frac{f(h^{-1}(t))}{\lambda}\right|^{p(h^{-1}(t))} \ dm(t)\\
=\int_{ I} \omega' (t)\left| \frac{(Tf)[t]}{\lambda}\right|^{r(t)} \ dm(t)
\end{multline*}
(possibly infinite) holds, where the Radon-Nikodym derivative $\frac{d(\mu \circ h^{-1})}{dm}\equiv 1$ can be disregarded. Thus it is easy to see that the above $T$ is a linear isometry   
$L^{p(\cdot)} (\R^n, \mathrm{Bor}(\R^n), m_n , \omega) \to L^{r(\cdot)} (I,\mathrm{Bor}\left(I\right) , \omega')$. In fact this is an isomorphism, surjectivity is easiest to check by 
noting that $T$ has an inverse given by the formula 
\[T^{-1} (g) [x] = g[h(x)].\]
\end{proof}

Moreover, with a similar reasoning as above, any Musielak-Orlicz space on 
$(\R^n , \mathrm{Bor}(\R^n ), m_n)$  allows a dimension reduction to the real line. 
Next we will reduce the measure space of our function space to the unit interval. 
We will consider the more complicated case with $\mu(\R^n)=\infty$, if $\mu$ is finite then a similar fact holds. 

\begin{proposition}\label{prop: }
Let $L^{r(\cdot)} ([0,\infty), \mathrm{Bor}\left([0,\infty) \right) , \omega')$ be as above and 
$h\colon \left[0,\infty\right)\to [0,1)$, $h(t)=1- \frac{1}{1+t}$. Then 
\[(S f)[h(t)] =(1+t)^{\frac{2}{r(t)}} f[t] \quad \text{for}\ m\text-{a.e.}\ t\in [0,\infty)\] 
defines an isometric isomorphism
\[T\colon L^{r(\cdot)} ([0,\infty), \mathrm{Bor}\left([0,\infty) \right) , m, \omega')
\to L^{q(\cdot)} ([0,1], \mathrm{Bor}\left([0,1] \right) , m, w) \] 
where 
\[q(h(t)) =r(t)\quad \text{and} \quad     w (h(t)) = \omega' (t) \quad \text{for}\ m\text{-a.e.}\ t\in [0,\infty) .\]
\end{proposition}
We mainly omit the proof, just note that after taking the $q(h(t))=r(t)$ power in the target space the weight 
becomes $(1+t)^{2}$, this is the inverse of the Radon-Nikodym derivative of the measure transformation 
(plainly $|h' (t)|$) which is required as a 'compensator' in the change of variable.  
\medskip

Next we comment on the weight appearing in the Nakano norm. On one hand, this weight corresponds to a special state
in our ODE of interest, where the value of the solution $\varphi (t)=1$ has a 'neutral' effect on its own growth.
On the other hand, some form of weight compensation is required for the continuity of the modular. 

\subsection{Continuity of the modulars}\label{sect: CM}
Note that the map
\[\lambda \mapsto \int_{0}^1 \left|\frac{f(t)}{\lambda}\right|^{p}\ dt,\quad \lambda>0 \]
in the MO norm expressing the classical $L^p$ norm is in fact continuous with respect to $\lambda$. 
This appears more generally a reasonable requirement on the modulars, say, at least for bounded functions $f$. 
\begin{proposition}\label{prop: lambda}
We investigate below measurable functions $f\colon [0,1]\to \R$, $p\colon [0,1]\to [1,\infty)$, 
$w\colon [0,1]\to (0,\infty)$ and reals $\lambda>0$. Consider the following conditions:
\begin{enumerate}
\item There exists $c>1$ and $C>0$ such that 
\[w(t)\leq \frac{C}{c^{p(t)}}\quad \text{for\ a.e.}\ t\]
\item For each $f$ such that $\int_{0}^1 |f(t)|^{p(t)}\ dt \leq 1$ the mapping 
\[\lambda \mapsto \int_{0}^1 w(t)\ \frac{|f(t)|^{p(t)}}{\lambda^{p(t)}}\ dt\]
is continuous at $\lambda=1$.
\item There exists $D>0$ such that 
\[w(t)\leq \frac{D}{p(t)}\quad \text{for\ a.e.}\ t.\]
\end{enumerate}
Then $(1) \implies (2) \implies (3)$. Moreover, if $p$ is essentially bounded and $\int_{0}^1 |f(t)|^{p(t)}\ dt=1$, then 
\[\frac{d}{d\lambda} \int_{0}^1 \frac{1}{p(t)} \frac{|f(t)|^{p(t)}}{\lambda^{p(t)}}\ dt\bigg\vert_{\lambda=1}
=- 1.\]
\end{proposition}

The latter observation above does not fully motivate by itself the use of the weight $w$, since we made the ad hoc assumption that
$\int_{0}^1 |f(t)|^{p(t)}\ dt \leq 1$. However, continuing our heuristic line of reasoning, the constant functions should be
canonical enough a test bed for assessing the behavior of weight functions. We note that the continuity of 
\[\lambda \mapsto \int_{0}^1 w(t)\ \frac{1}{\lambda^{p(t)}}\ dt\]
at $\lambda=1$ implies 
\[\int_{0}^1 p(t) w(t)\ dt <\infty .\]
Here the value $\lambda=1$ is plausible, 
\[\int_{0}^1 w(t)\frac{1}{\lambda^{p(t)}}\ dt>1\ \Leftrightarrow\ \lambda<1, \]
for instance, if  
\[ \frac{C}{(p(t))^{\alpha}} \leq w(t) \leq 1\] 
for some constants $C,\alpha >0$ and $p$ tends to $\infty$ suitably slowly.

\subsection{Connections between the ODE-determined norm and the modular} 
We may express the ODE-determined norm in a manner similar to the Luxemburg norm. Admittedly, this involves choosing the weight in a very liberal way, somewhat in the same style as in Proposition \ref{prop: lambda}. We 
may write

\begin{eqnarray*}
\|f\|_{p(\cdot)}&=&\varphi_f (1)\\
&=& \inf\left\{\lambda>0\colon \frac{1}{\lambda} \varphi_f (1)\leq 1\right\}\\
&=& \inf\left\{\lambda>0\colon \frac{1}{\lambda}\int_{0}^1 \frac{|f(t)|^{p(t)}}{p(t)} \varphi_{f}^{1-p(t)} (t) \ dt\leq1\right\}\\
&=& \inf\left\{\lambda>0\colon \int_{0}^1 \frac{(\varphi_{f}(t)/\lambda )^{1-p(t)}}{p(t)}\ (|f(t)|/\lambda)^{p(t)}\ dt\leq 1\right\}\\
&=&\inf\left\{\lambda>0\colon \int_{0}^1 \frac{1}{p(t)(\varphi_{f/\lambda}(t))^{p(t)-1}}\ (|f(t)|/\lambda)^{p(t)}\ dt\leq1\right\}.\\
\end{eqnarray*}

Here 
\[\frac{d}{d\lambda} \int \frac{1}{p(t)(\varphi_{f/\lambda}(t))^{p(t)-1}}\ (|f(t)|/\lambda)^{p(t)}\ dt
=-\frac{1}{\lambda^2}\int \frac{|f(t)|^{p(t)}}{p(t)} \varphi_{f}^{1-p(t)} (t) \ dt\]
and for $\lambda=\|f\|_{p(\cdot)}$ the above reads $=- \frac{1}{\|f\|_{p(\cdot)}}$.
So, in this case the modular does not merely define the norm by means of a level set, but it actually behaves locally 
according to the required norm. 

Approaching the connection between ODE-determined norms and Luxemburg norms from another direction, suppose that for some weight function $w(t)>0$ and scalar $\lambda>0$ we have
\[\int_{0}^1 w(t) \frac{|f(t)|^{p(t)}}{\lambda^{p(t)}}\ dt=1.\]
This can be rewritten as 
\[\int_{0}^1 w(t)\ |f(t)|^{p(t)}\  \lambda^{1-p(t)}\ dt=\lambda.\] 
Replacing $f$ with $1_{[0,t]}f$, $0\leq t \leq 1$, leads to separate respective solutions
$\lambda_t$ with
\[\lambda_t= \int_{0}^{t} w(s)\ |f(s)|^{p(s)}\  (\lambda_t )^{1-p(s)}\ ds .\]
Heuristically speaking, the scalars $\lambda_t$ may be considered some kind of averages of more localized constants or a function with a similar role. Let us further localize these scalars by defining 'a varying lambda' function $\lambda(t)$ as the solution to 
\[\lambda(t) = \int_{0}^{t} w(s)\ |f(s)|^{p(s)}\  (\lambda(s))^{1-p(s)}\ ds ,\]
if such a weak solution exists, thus with a weak formulation
\begin{equation}\label{eq: lambda}
\lambda'(t) = w(t)\ |f(t)|^{p(t)}\  (\lambda(t))^{1-p(t)}\ \text{a.e.}
\end{equation}
Let us see what happens if $p(t)=p$ is a constant and we choose $w(t)=\frac{1}{p}$. Then
\[p\ (\lambda(t))^{p-1} \ \lambda'(t) = |f(t)|^{p(t)}\ \text{a.e.}\]
This yields
\begin{eqnarray*}
\int_{0}^t p\ (\lambda(s))^{p-1} \ \lambda'(s)\ ds &=& \int_{0}^t  |f(s)|^{p}\ ds,\\
(\lambda(t))^p &=& \int_{0}^t  |f(s)|^{p}\ ds\\
\lambda(t) &=& \|1_{[0,t]} f\|_p .
\end{eqnarray*}
This is compatible with the philosophy of how the $\lambda$:s are defined and used,
\[\left\|\frac{1_{[0,t]} f}{\lambda(t)}\right\|=1.\]

The conclusion is that our solutions $\varphi_f$ and the 'varying lambdas' coincide when the weight in \eqref{eq: lambda} is chosen to be $w(t)=\frac{1}{p(t)}$.

\subsection*{Acknowledgments}

This work has received financial support from the V\"ais\"al\"a foundation, the Finnish Cultural Foundation and the Academy of Finland Project \# 268009.

\end{document}